\documentclass[12pt]{amsart}

\usepackage{amsfonts, amstext, amsmath, amsthm, amscd, amssymb, graphicx, epstopdf}
\usepackage{epsfig, graphics, psfrag, enumerate}
\usepackage{color}
\usepackage{verbatim}
\let\oldmarginpar\marginpar
\renewcommand\marginpar[1]{\oldmarginpar[\raggedleft\footnotesize #1]%
{\raggedright\footnotesize #1}}

\newtheorem{theorem}{Theorem}[section]

\newtheorem{lemma}[theorem]{Lemma}
\newtheorem{corollary}[theorem]{Corollary}

\newtheorem{definition}[theorem]{Definition}

\theoremstyle{definition}

\newtheorem{example}[theorem]{Example}
\newcommand{\vtet}{{v_{\rm tet}}}
\newcommand{\voct}{{v_{\rm oct}}}

\newcommand{\ZZ}{{\mathbb{Z}}}

\newcommand{\guts}{{\rm guts}}
\newcommand{\cut}{{\backslash \backslash}}

\newcommand{\vol}{{\rm vol}}

\newcommand{\abs}[1]{{\left\vert #1 \right\vert}}

\newcommand{\GA}{{\mathbb{G}_A}}
\newcommand{\GB}{{\mathbb{G}_B}}
\newcommand{\G}{{\mathbb{G}}}
\newcommand{\GRA}{{\mathbb{G}'_A}}
\newcommand{\GRB}{{\mathbb{G}'_B}}

\newcommand\no[1]{}

\newtheorem*{namedtheorem}{\theoremname}
\newcommand{\theoremname}{testing}

\def\be { \begin{equation} }
\def\ee { \end{equation} }

\begin{document}

\title{State Surfaces of Links}

\author{Efstratia Kalfagianni}
\address[]{Department of Mathematics, Michigan State University, East
Lansing, MI, 48824, USA}

\email[]{kalfagia@math.msu.edu}
\thanks{The author is supported in part by NSF grants DMS-1404754 and DMS-1708249.}

%\begin{document}
%\begin{abstract}State surfaces are spanning surfaces of links that  are obtained from link diagrams  guided by the combinatorics underlying  Kauffman's  construction of the Jones link polynomial
%via \emph{state models} \cite{states}. Geometric properties of such surfaces are often dictated  by simple  link diagrammatic criteria  and can be used to study geometric structures of link 
%complements. State surfaces also provide a tool for establishing relations between Jones polynomials and topological link invariants, such as the crosscap number or   invariants coming from geometric structures on link complements (e.g. hyperbolic volume).This is important as it sheds some light on the long standing open question about  the geometric link properties detected by Jones polynomials.
%In this article
%we survey the construction of state surfaces of links and some of their recent applications.
%\end{abstract}

\maketitle

\section{Introduction}\label{Sec:preface}

State surfaces are spanning surfaces of links that  are obtained from link diagrams. Their construction  is  guided by the combinatorics underlying  Kauffman's  construction of the Jones link polynomial
via \emph{state models}. Geometric properties of state surfaces are often dictated  by simple  link diagrammatic criteria, and the surfaces themselves carry important  information about geometric structures of link 
complements. On the other hand, certain state surfaces carry spines (state graphs) that can be used to compute the Jones polynomial of links. From this point of view,
state surfaces  provide a tool for establishing relations between Jones polynomials and topological link invariants, such as the crosscap number or   invariants coming from geometric structures on link complements (e.g. hyperbolic volume). In this article
we survey the construction of state surfaces of links and some of their recent applications.

%%%%%%%%%%%

\section{Definitions and examples}\label{Sec:Intro}

For a link $K$ in $S^3$,  $D = D(K)$ will denote a link diagram, in the equatorial
$2$--sphere of $S^3$. We  will often abuse by referring to the projection
$2$--sphere using the common term projection plane. In
particular, $D(K)$ cuts the projection ``plane'' into compact regions each of which is a polygon with vertices at the crossings of $D$.

Given a crossing on a  link diagram $D(K)$ there are two ways to resolve
it; the $A$- resolution and the $B$-resolution  as shown in Figure \ref{resolve}. The figure is borrowed from \cite{KaLee}.  Note that if the link $K$ is oriented,  only one of the two resolutions at each crossing
will respect the orientation of $K$. 
A Kauffman state $\sigma$ on $D(K)$ is a choice of one of these two resolutions at each crossing of $D(K)$ \cite{states}.
  For each state $\sigma$  of a  link diagram the  \emph {state graph} ${\G}_{\sigma}$  is constructed as follows: The result of applying $\sigma$
to $D(K)$
is a collection $v_{\sigma}(D)$ of non-intersecting
circles in the plane,  called \emph{state circles}, together with embedded arcs  recording the crossing
splice. 
Next we obtain the  \emph{state surface} $S_{\sigma}$, as follows: Each  circle of $v_{\sigma}(D)$ bounds a disk in $S^3$. This collection of disks can be disjointly embedded in the ball below the projection plane.
At each crossing of $D(K)$, we connect the pair of neighboring disks by a half-twisted band to construct a surface $S_{\sigma} \subset S^3$ whose boundary is $K$. 
\begin{figure}[ht]
\includegraphics[scale=.6]{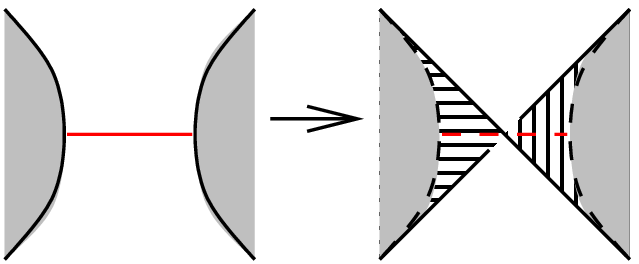}
\hspace{2cm}
\includegraphics[scale=.6]{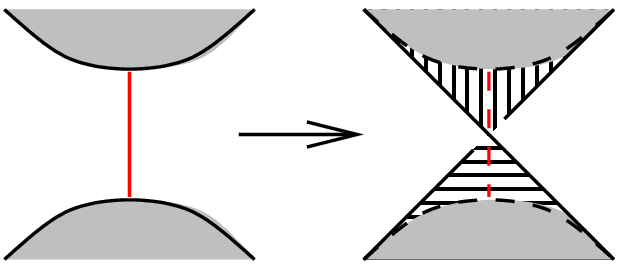}
\label{resolve}
\caption{The $A$-resolution (left), the $B$-resolution (right)  of a crossing and their contribution to state surfaces.
} 
\end{figure}

\begin{example} \label{AB}{\rm{Given an oriented link diagram $D=D(K)$, the Seifert state, denoted by $s(D)$,  is the one that assigns to each crossing of $D$ the resolution that is consistent with the orientation of $D$.
The corresponding state surface $S_s=S_s(D)$ is  oriented (a.k.a. a Seifert surface). The process of constructing $S_s$ is known  a {\emph{Seifert's algorithm}} \cite{Lickbook}.

By applying the
$A$--resolution to each crossing of $D$, we obtain a
crossing--free diagram $s_A(D)$.
Its state graph, denoted by
$\GA=\GA(D)$,  is called  the all--$A$ state  graph and the corresponding state surface is denoted by $S_A=S_A(D)$. An example is shown in  Figure  \ref{fig:statesurface},
 which is borrowed from \cite{GutsBook}. Similarly, for the all--$B$ state the crossing--free resulting diagram is denoted by
$s_B(D)$, the state
graph is denoted $\GB$, and the state surface by $S_B$.}}\end{example}

\begin{figure}
\includegraphics[scale=.8]{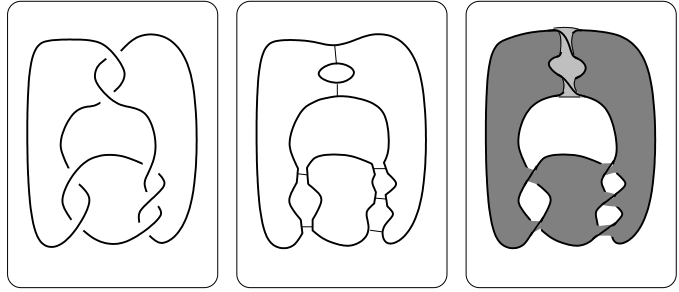}
\caption{Left to right:  A diagram, the all-$A$ state graph $\GA$ and the corresponding   state
	surface $S_A$.}
\label{fig:statesurface}
\end{figure}

By construction, $\mathbb{G}_\sigma$ has one vertex for every circle of
$v_\sigma$ (i.e. for every  disk in $S_\sigma$), and one edge for every
half--twisted band in $S_\sigma$. This gives a natural embedding of
$\mathbb{G}_\sigma $ into the surface, where vertices are  embedded
into the corresponding disks, and edges run through the
corresponding half-twisted bands. Hence,  $\mathbb{G}_\sigma$ is a spine for  $S_\sigma$.

\begin{lemma}\label{lemma:ga-spine}
The surface $S_\sigma$ is orientable if and only if
$\mathbb{G}_\sigma$ is a bipartite  graph.
\end{lemma}

\begin{proof}

Recall that a graph is bipartite if and only if all cycles (i.e. paths from any vertex to itself)
contain an even number of edges.

If $\mathbb{G}_\sigma$ is
bipartite, we may assign an  orientation on
$S_\sigma$, as follows: Pick a normal direction to one disk,
corresponding to a vertex of $\mathbb{G}_\sigma$, extend
over half--twisted bands to orient every adjacent disk, and continue
inductively. This inductive process
$S_\sigma$ will not run into a contradiction since
every cycle in $\mathbb{G}_\sigma$ has even number of edges. Thus $S_\sigma$ is
a two--sided surface in $S^3$, hence orientable. This is the case with the example of Figure \ref{fig:statesurface}.

Conversely, suppose $\mathbb{G}_\sigma$ is not
bipartite, hence contains a cycle with an odd number of edges. By embedding
$\mathbb{G}_\sigma$ as a spine of $S_\sigma$, as above, we see that this cycle is an
orientation--reversing loop in $S_\sigma$.\qed
\end{proof}

%%%%%%%%%%%%%%%%%%%%%%%%%%%%%%%%%%%%%%%%%%%%%%%%%%%%%%%%%%%%%%%%%
\section{Genus and crosscap number of alternating links }\label{Sec:genus}
The genus of an orientable surface  $S$ with with $k$ boundary components is defined to be $1-(\chi(S)+k)/2$,
where $\chi(S)$ is the Euler characteristic of $S$ and the  \emph{crosscap number} of a non-orientable surface with $k$ boundary components is defined to be
$2-\chi(S)-k$.

\begin{definition}Every link in $S^3$ bounds both orientable and non-orientable surfaces. 
The  \emph{genus} of an oriented  link $K$, denoted by $g(K)$, is  the minimum genus  over all orientable surfaces  $S$ 
bounded by $K$. That is we have $\partial S=K$.
The \emph{crosscap number}  (a.k.a. non-orientable genus) of a link $K$, denoted by $C(K)$,  is the minimum crosscap number over all non-orientable surfaces 
spanned by $K$.
\end{definition}

For \emph{ alternating links}   the genus and the crosscap number
can be  computed using  state surfaces of alternating link diagrams. For the orientable case, 
we recall the following classical result due to Crowell \cite{crowell} (see also \cite{Lickbook}).
\begin{theorem} \cite{crowell}  Suppose that $D$ is a connected alternating diagram of a $k$-component link $K$. Then the state surface $S_s(D)$ corresponding to the Seifert state of $D$ realizes the 
genus of $K$. That is we have $g(K)=1-{(\chi(S_s(D))+k)} /2$.
\end{theorem}

In  \cite{Adamsstate}, Adams and Kindred used state surfaces  to give an  algorithm for computing crosscap numbers of alternating links.
To summarize  their  algorithm and state their result, consider a connected alternating diagram $D(K)$ as 4-valent a graph on $S^2$. 
Each region in the complement of the graph is an $m$-gon 
with vertices at the vertices of the graph.

\begin{lemma} \label{triangle} Suppose that $D(K)$ is a connected alternating link diagram whose complement has no
 bigons or 1-gons. Then  at least one region must be a triangle.
\end{lemma}
 \begin{proof}
Let $V$, $E$, $F$ 
denote the number of vertices, edges and complimentary regions of $D(K)$, respectively.
 Then,
$V - E + F = 2$ and
 $E = 2V$, which implies that $F > V$. Suppose that none of the $F$ regions is a triangle.
 Then,   $F < 4V / 4 = V$ since each region  has  at least four  vertices and each vertex can only be on at most 4 distinct regions. This is a contradiction. \qed \end{proof}

 Observe that the Euler characteristic of a surface,
corresponding to a state $\sigma$, is      $\chi(S_{\sigma})=v_{\sigma}-c$, where $c$ is the number of crossings on $D(K)$.
Thus to maximize $\chi(S_{\sigma})$ we must maximize the number of state circles $v_{\sigma}$.
Now we outline the algorithm from  \cite{Adamsstate} that finds a surface of maximal Euler characteristic (and thus of minimum genus) over all surfaces (orientable and non-orientable)  spanned by an alternating link.
\vskip 0.08in

\noindent { {\bf{Adams-Kindred algorithm:}}} \ Let $D(K)$ be a connected, alternating diagram.
\begin{enumerate}
\item Find the smallest $m$ for which the complement of the projection $D(K)$ contains an $m$-gon. 

\item If $m=1$, then we resolve the corresponding crossing so that the $1$-gon becomes a state circle.

Suppose that $m=2$. Then  some regions of  $D(K)$  are bigons.
Create one branch of the algorithm for each bigon on $D(K)$.
Resolve the two crossings corresponding to the vertices of the bigon
so that the bigon is bounded by a state circle. See Figures 1.4 and \ref{fstep2}  below. \label{step2a}

\vskip 0.02in

\item Suppose $m>2$. Then  by Lemma  \ref{triangle}, we have $m = 3$. Pick a triangle region on $D(K)$.
Now the process has two branches: For one branch 
we resolve each crossing on the triangle's boundary so that the triangle becomes a state  circle. For the other branch, we resolve each of the crossings the opposite way.

\begin{figure}[ht]
\centering
\includegraphics[scale=.15]{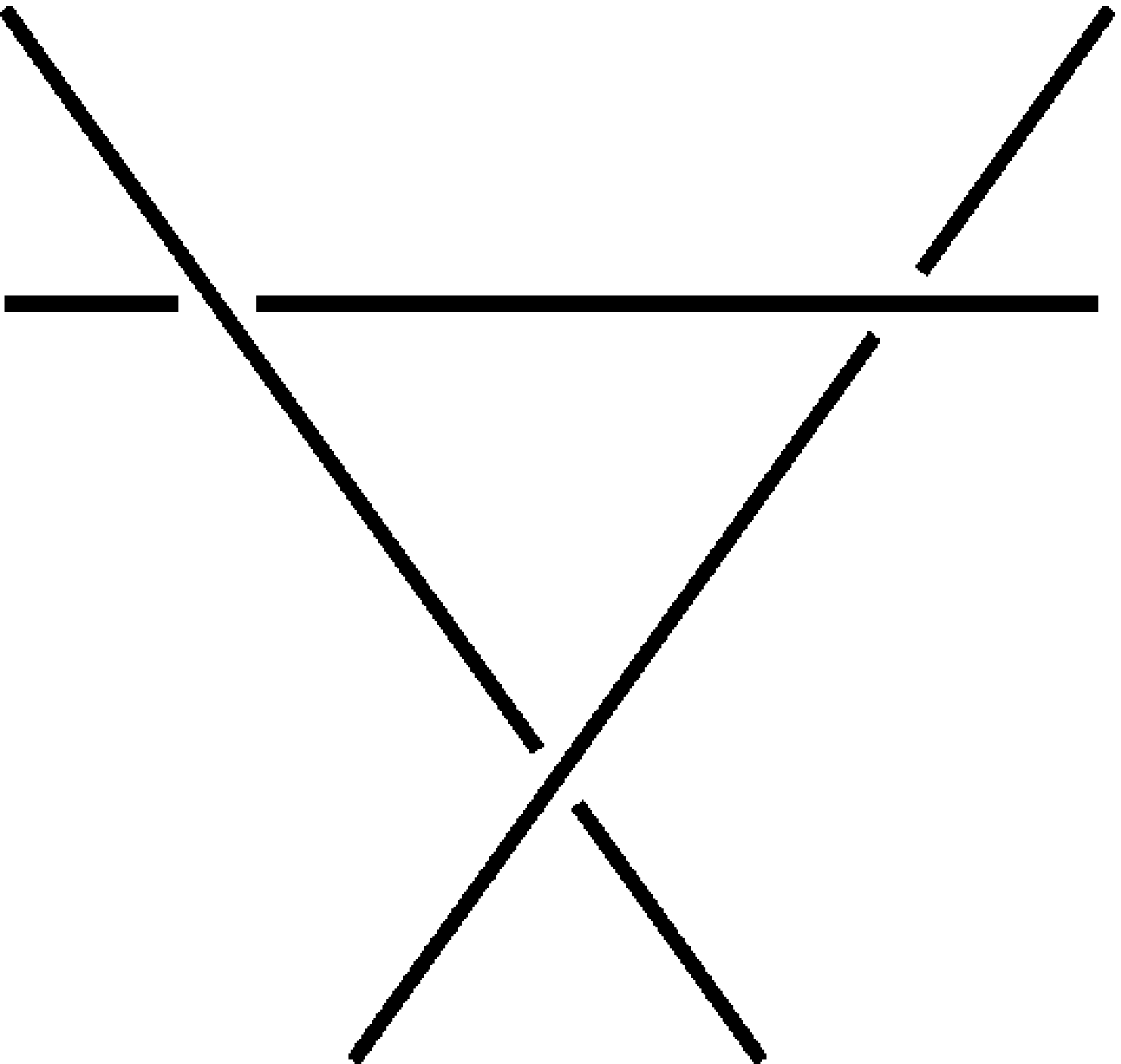}
\hspace{1cm}
\includegraphics[scale=.15]{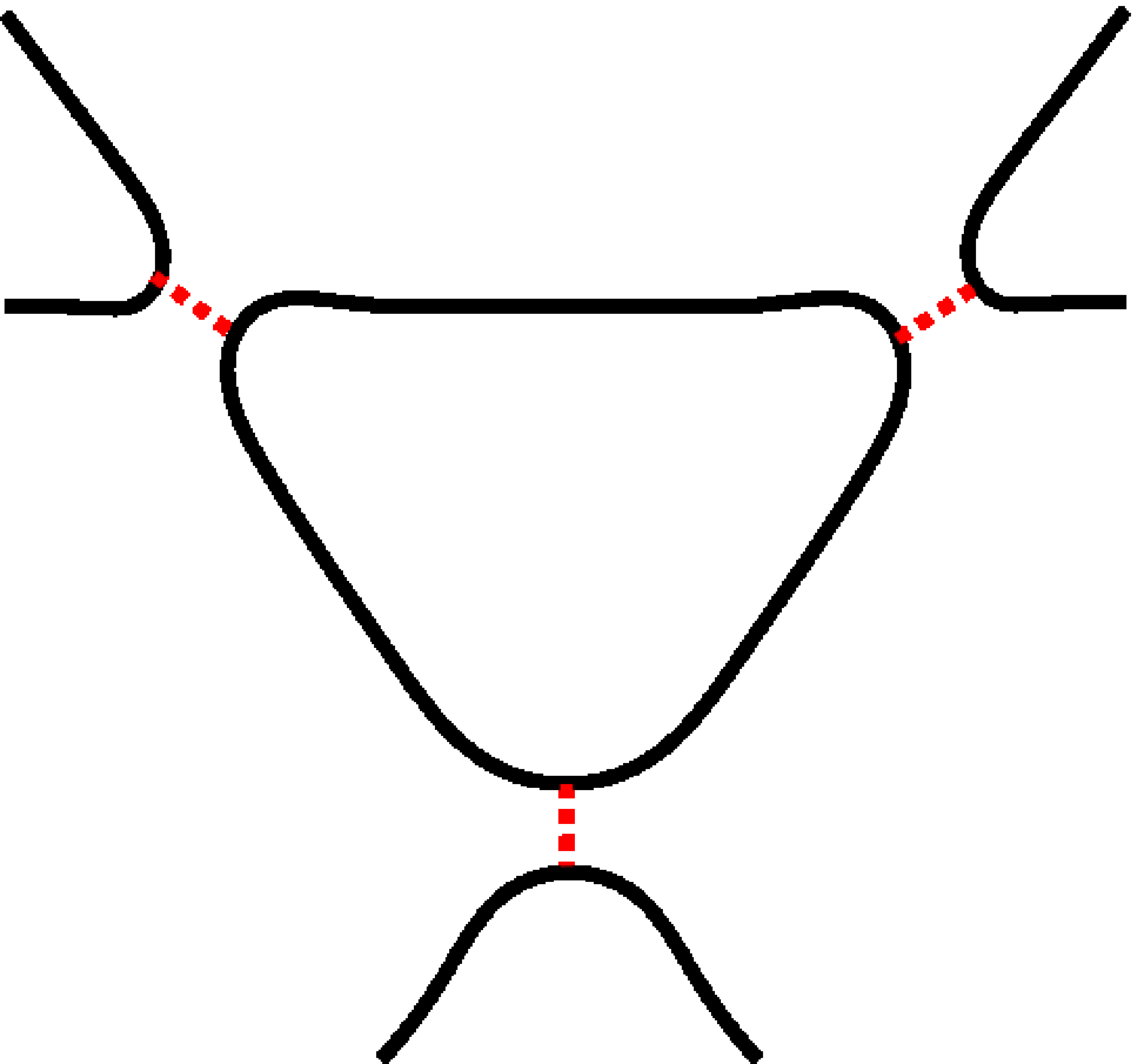}
\hspace{1cm}
\includegraphics[scale=.15]{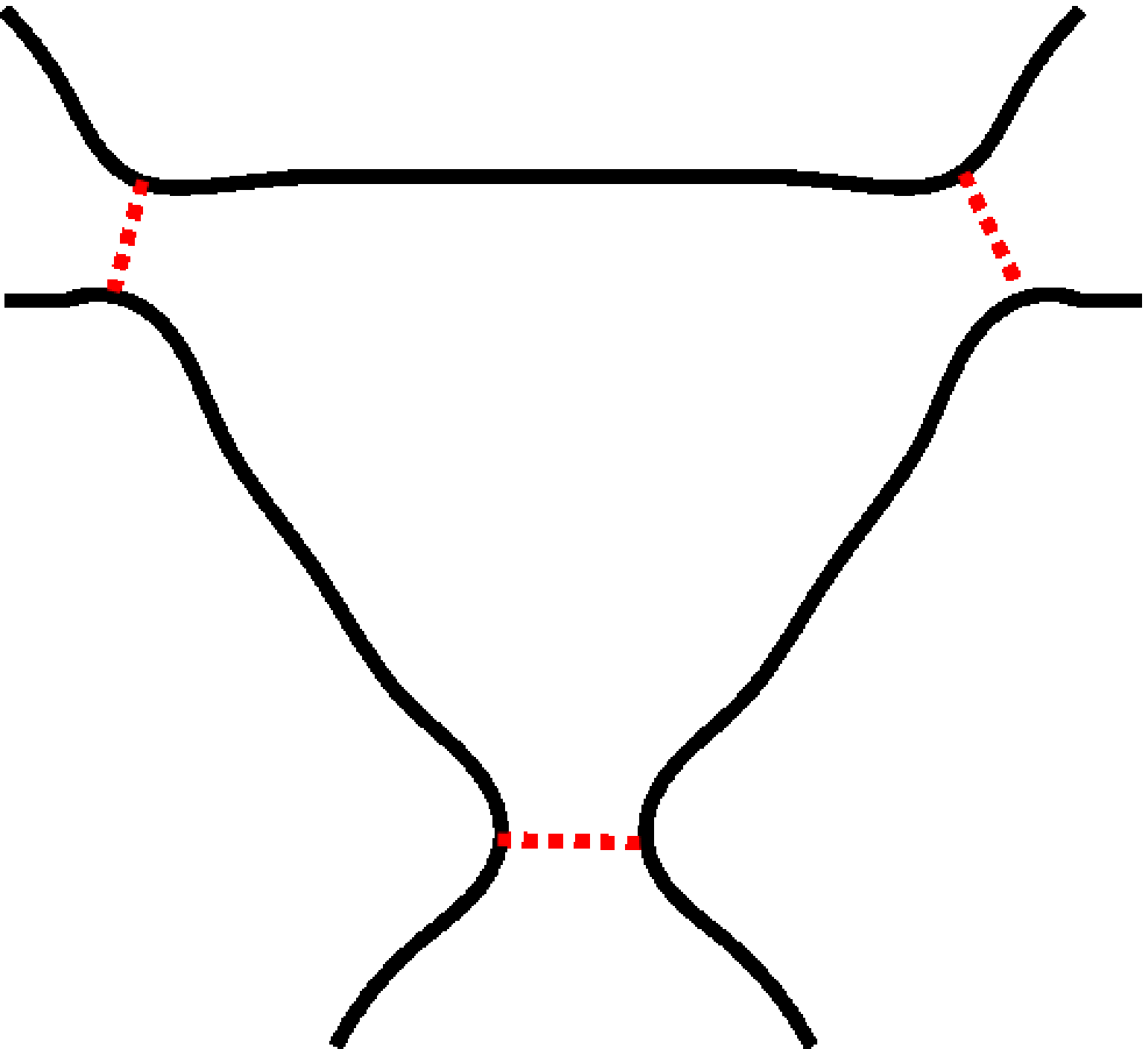}
\caption{The two branch of the algorithm for triangle regions. The  is figure borrowed from
\cite{KaLee}.
}
\label{3-gon}
\end{figure}

\item Repeat  Steps 1  and 2 until each branch reaches a projection without crossings. Each branch corresponds 
to a Kauffman state of $D(K)$  for which there is a corresponding state surface.
Of all the branches 
involved in the process choose  one that has the largest number of state circles. The surface $S$ corresponding to this 
state has maximal Euler characteristic over all the states corresponding to $D(K)$.
Note that, {\emph{a priori}}, more than one branches of the algorithm may lead to surfaces of maximal Euler characteristic. 

\end{enumerate}

\begin{theorem}{\cite{Adamsstate}}\label{AK}  
Let $S$ be any maximal Euler characteristic  surface obtained via above algorithm  from an alternating diagram of $k$-component link $K$.
Then, 
\begin{enumerate}
\item If there is a surface $S$ as above that is non-orientable then $C(K)=2-\chi(S)-k$.
\vspace{0.1in}
\item  If all the surfaces $S$ as above are orientable, we have $C(K)=3-\chi(S)-k$.
Furthermore, $S$ is a minimal genus Seifert surface of $K$ and  $C(K)=2g(K)+1$.
\end{enumerate}
\end{theorem}
\smallskip

\begin{example}\label{Fig8} {\rm{Different choices of branches as well as the order in resolving bigon regions following the algorithm above, may result in different state surfaces. In particular at the end of the algorithm  we may
have  both orientable and non-orientable surfaces that share the same Euler characteristic:

\begin{figure}[ht]
\centering
\includegraphics[scale=.3]{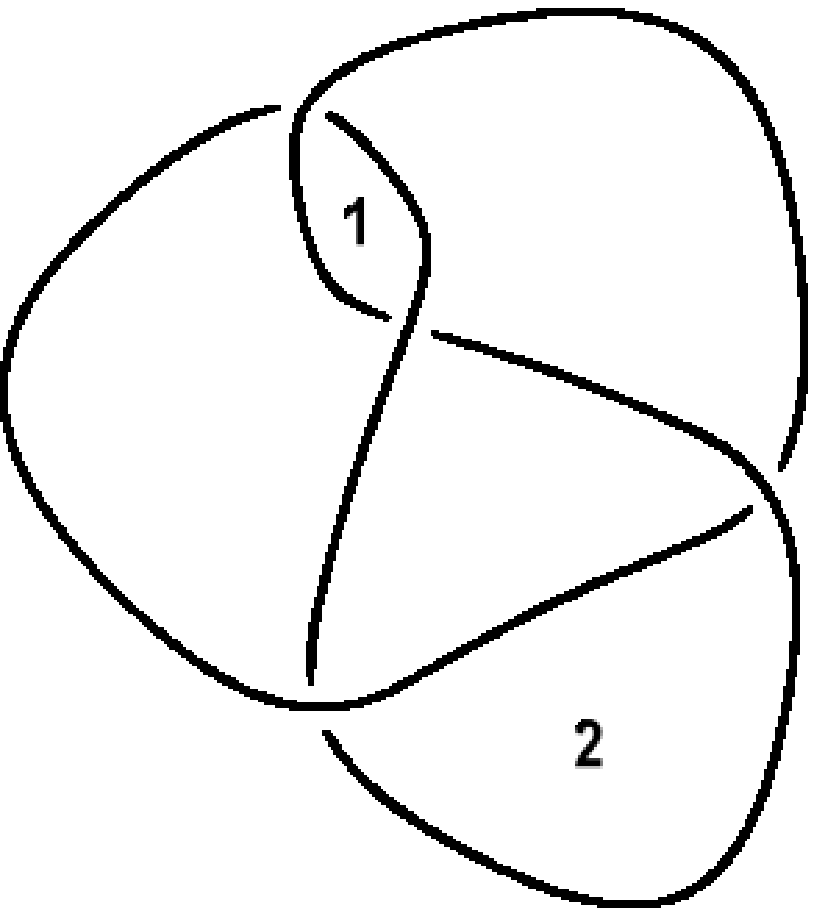} 
\hspace{2cm}
\includegraphics[scale=.3]{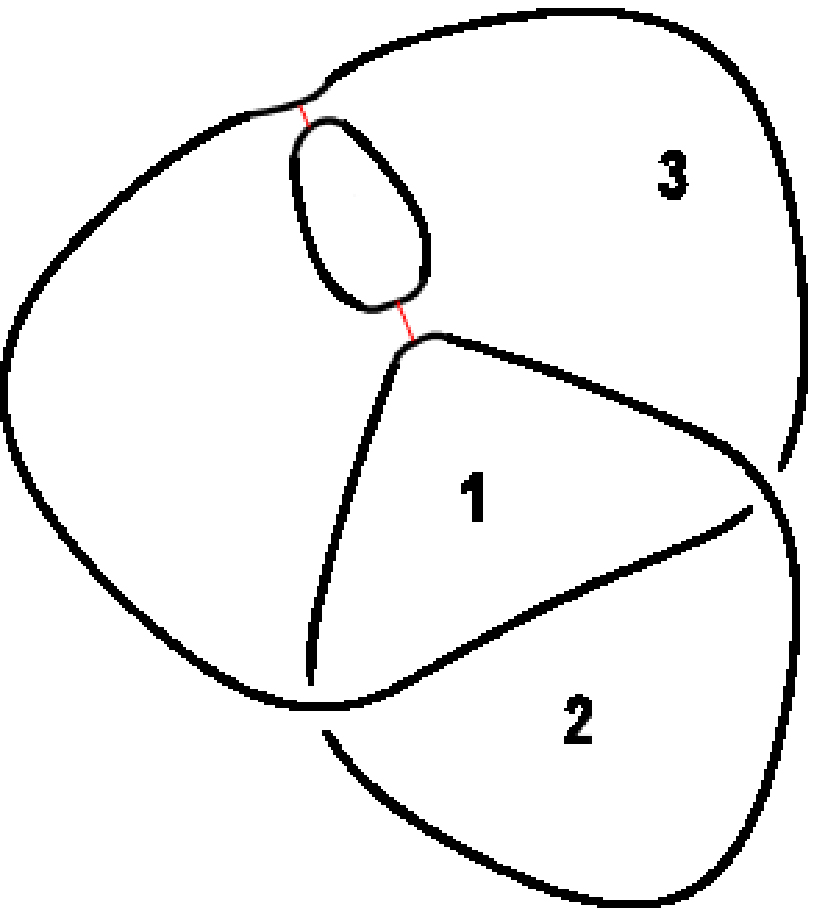}
\label{fstep0}
\caption{A diagram of   $4_1$ with bigon regions 1 and 2 and the result of applying step \ref{step2a} of the algorithm  to  bigon 1.}
\end{figure}

 Suppose that we choose the bigon labeled by  1 in the left hand side picture of  Figure 1.4.
Then, for the next step of the algorithm, we have three choices of bigon regions to resolve, labeled by 1 and 2 and 3 of the figure.

\begin{figure}[ht]
\includegraphics[scale=.3]{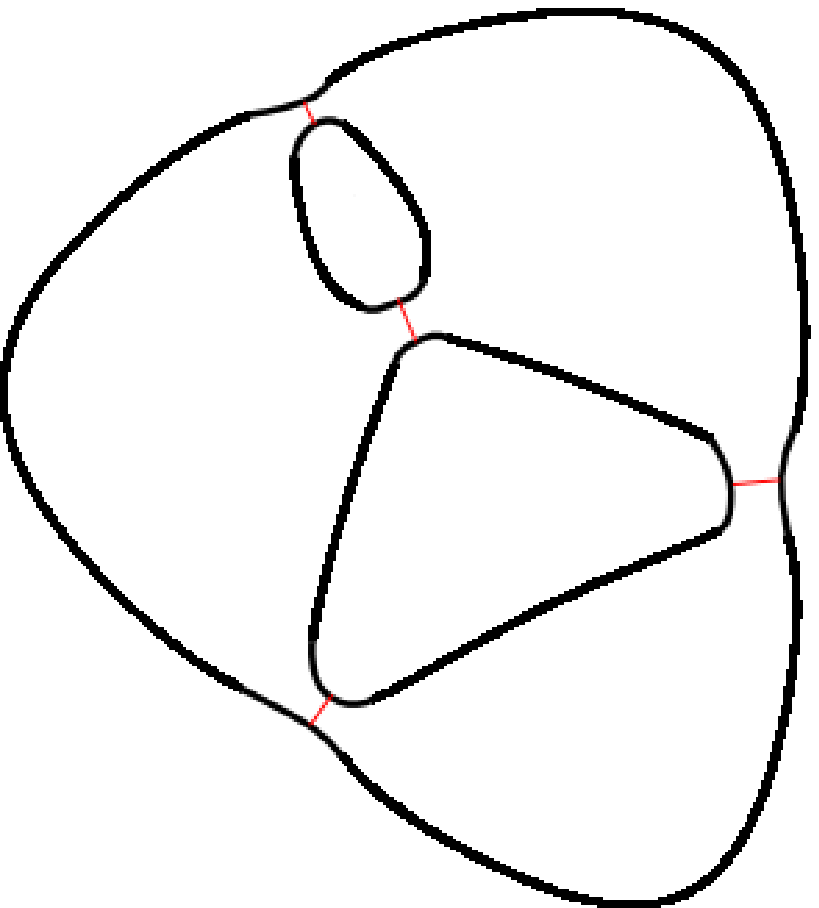} 
\hspace{2cm} 
\includegraphics[scale=.3]{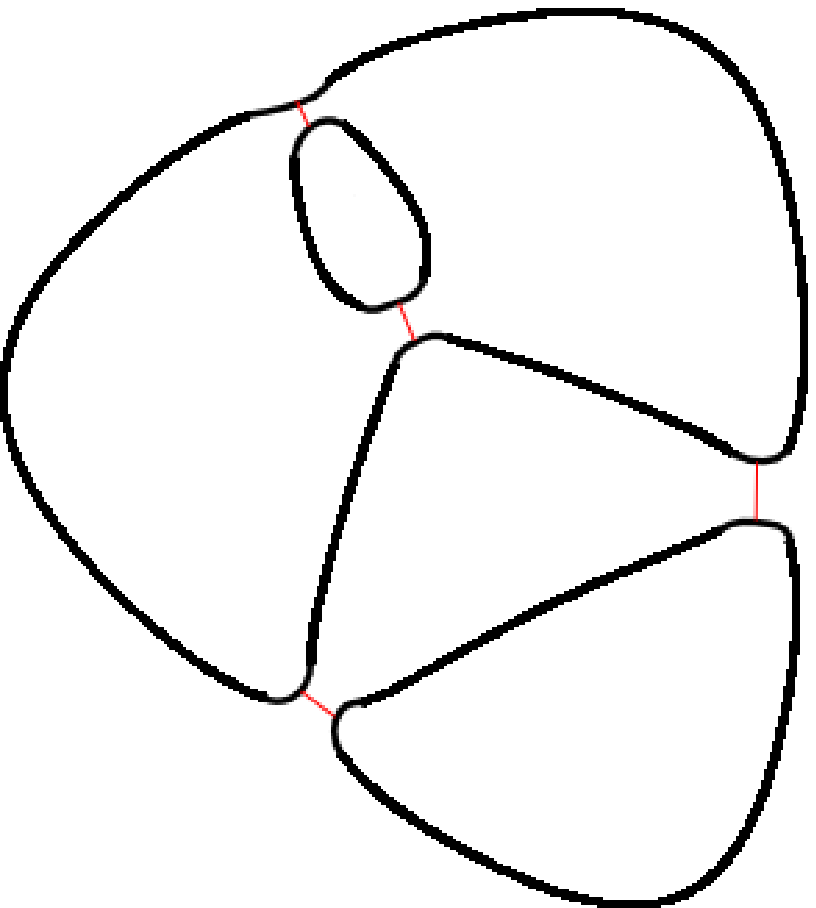}
\caption{Two algorithm branches corresponding to different bigons.}
\label{fstep2}
\end{figure}

The choice of bigon 1 leads 
to a non-orientable surface, shown in the left panel  of Figure \ref{fstep2}, realizing  the crosscap number of $4_1$, which is two.
The choice of bigon 2 leads to an orientable surface, shown in the right panel  of Figure \ref{fstep2}, realizing the genus of the knot which is one.
Both surfaces realize the maximal Euler characteristic of -1.}}
 \end{example}

%%%%%%%%%%%%%%%%%%%%%%%%%%%%%%%%%%%%%%%%%%%%%%%%%%%%%%%%%%%%%%%%%
\section{Jones polynomial and state graphs}\label{Sec:Jones}

A connected link diagram $D$ defines a  4--valent planar graph
 $\Gamma \subset S^2$, 
which leads to the construction of the \emph{Turaev surface} $F(D)$ as follows \cite{dasbach-futer...}:
  Thicken the projection plane to $S^2 \times
[- 1, 1]$, so that $\Gamma$ lies in $S^2 \times \{0\}$. Outside a
neighborhood of the vertices (crossings) the  surface intersect
$S^2 \times
[- 1, 1]$, in $\Gamma \times [- 1, 1]$. In the neighborhood of
each vertex, we insert a saddle, positioned so that the boundary
circles on $S^2 \times \{1\}$ are the
components
of the $A$--resolution $s_A(D)$, and the boundary circles on $S^2
\times \{- 1\}$ are the components of $s_B(D)$.

When $D$ is an alternating diagram, each circle of
$s_A(D)$ or $s_B(D)$ follows the boundary of a region in the
projection plane. Thus, for alternating diagrams, the surface $F(D)$
is  the projection sphere $S^2$. For general diagrams, the  diagram $D$  still is alternating on 
$F(D)$.

The surface $F(D)$ has a natural
cellulation: the $1$--skeleton is the graph $\Gamma$ and the 
$2$--cells correspond to circles of $s_A(D)$ or $s_B(D)$, hence to
vertices of $\GA$ or $\GB$. These $2$--cells admit a 
checkerboard coloring, in which the regions corresponding to the
vertices of $\GA$ are white and the regions corresponding to $\GB$ are
shaded. The graph $\GA$ (resp. $\GB$) can
 be embedded in $F(D)$ as the
adjacency graph of white (resp. shaded) regions.
The \emph{faces} of $\GA$ (that is, regions in the complement of $\GA$) correspond to vertices of $\GB$, and vice versa. Hence the graphs are dual to one another on $F(D)$.
Graphs, together with such embeddings into an orientable surface,  called  \emph{ribbon graphs} have been studied in the literature  \cite{BR}. Building on this point of view,
 Dasbach, Futer, Kalfagianni, Lin and Stoltzfus
\cite{dasbach-futer...} showed that the ribbon graph embedding of $\GA$ into the Turaev surface $F(D)$ carries at least as much information as the Jones polynomial $J_K(t)$.
To state the relevant result from \cite{dasbach-futer...},
recall that
  a \emph{spanning} subgraph of $\GA$ is a subgraph that contains all
  the vertices of $\GA$.  Given a spanning subgraph $\G$ of $\GA$ we
  will use $v(\G)$, $e(\G)$ and $f(\G)$ to denote the number of
  vertices, edges and faces of $\G$ respectively.

\begin{theorem}\cite{dasbach-futer...}\label{JPgraph} For  a connected link diagram $D$, the Kauffman bracket $\langle D \rangle\in
\ZZ[A, A^{-1}]$ is expressed as
$$\langle D \rangle = \sum_{\G\subset 
  \GA}^{\phantom{a}}\ A^{e(\GA)-2e(\G)} (-A^2-A^{-2})^{f(\G)-1},$$
where $\G$ ranges over all the spanning subgraphs of $\GA$.
\end{theorem}

Given a diagram $D=D(K)$, the {\emph{Jones polynomial}} of $K$, denoted by $J_K(t)$, is
obtained from $\langle D \rangle$  as follows: Multiply $\langle D
\rangle$ by  $(-A)^{-3w(D)}$, where $w(D)$ is the \emph{writhe} of
$D$, and then substitute $A = t^{-1/4}$ \cite{states, Lickbook}.

 Theorem \ref{JPgraph}  leads to formulae for the coefficients of $J_K(t)$
 in terms of topological quantities
of  the state graphs $\GA$ , $\GB$ corresponding to any diagram of $K$
\cite{dasbach-futer..., dfkls:determinant}. 
These formulae become
particularly effective if $\GA, \GB$ contain no 1-edges loops. In particular, this is the case when $\GA, \GB$ correspond to an alternating diagram that is \emph{reduced} (i.e. contains no redundant crossings).

\begin{corollary} \cite{dfkls:determinant} \label{coeffs} Let $D(K)$ be a reduced alternating  diagram and let  $\beta_K$ and $\beta'_K$ denote the second and penultimate  coefficient of $J_K(t)$, respectively . Let $\GRA$ and $\GRB$ denote the \emph{simple} graphs obtained by removing all duplicate edges between pairs of vertices of 
$\G_A(D)$ and $\G_B(D)$.  Then,
$$\abs{\beta_K}=1-\chi(\GRB), \ \ {\rm and} \ \  \abs{\beta'_K}=1-\chi(\GRA).$$
\end{corollary}

%%%%%%%%%%%%%%%%%%%%%%%%%%%%%%%%%%%

\section{Geometric Connections}\label{Sec:Geometry}
%\subsection{States, Graphs and surfaces}
To a link  $K$ in $S^3$  corresponds a  compact 3-manifold with boundary; namely $M_K=S^3\setminus N(K)$, where $N(K)$ is an open tube around $K$.
The interior of $M_K$ is homeomorphic to the link complement $S^3\setminus K$.
In the 80's,  Thurston \cite{thurston:notes} proved that link complements decompose canonically  into pieces that admit locally homogeneous geometric structures. 
A very common and interesting  case   is when the entire $S^3\setminus K$ has a hyperbolic structure, that is a metric of constant curvature $-1$ of finite volume.
 By Mostow rigidity, this hyperbolic structure is unique up to isometry, hence invariants of the  metric of $S^3\setminus K$   give topological invariants of $K$.
 
State surfaces obtained from link diagrams $D(K)$ give rise to properly embedded  surfaces in $M_K$. 
Many geometric properties of state surfaces can be checked through combinatorial and link diagrammatic criteria. 
For instance, Ozawa \cite{ozawa} showed that 
the all -$A$ surface $S_A(D)$ is  $\pi_1$--injective in $M_K$ if the state graph $\GA(D)$ contains no 1-edge loops. Futer, Kalfagianni and Purcell  \cite{GutsBook} gave a different proof of Ozawa's result
and also 
showed that $M_K$ is a fiber bundle  over the circle with fiber 
$S_A(D)$, if and only if the simple state graph $\GRA(D)$ is a \emph{tree}.

State 
surfaces have been used  to obtain relations between  combinatorial or Jones type link invariants and geometric invariants
of link complements.
Below we give a couple of  sample of such relations. For additional applications  the reader is referred to
 to \cite{AdamsCuspSizeBounds, BuKa, GutsBook, fkp:TAMS, lackenby:alt-volume, LPalte} and references therein.
The first result, proven combining  \cite{Adamsstate} with hyperbolic geometry  techniques, relates the crosscap number and  the Jones polynomial of alternating links.
 It was used to determine the crosscap numbers of 283 alternating knots of \emph{knot tables} that were  previously unknown \cite{Knotinfo}. 

\begin{theorem}\label{thm:cupjonesknots}\cite{KaLee} Given an an alternating, non-torus knot $K$,  with crosscap number $C(K)$,
 we have

$$ \left\lceil \frac{ T_K}{3}\right\rceil  + 1\; \leq \; C(K) \; \leq \;  {\rm {min}}{ \left\{ T_K + 1, \
\left\lfloor{ \frac{s_K}{2}}\,\right\rfloor \right\}}
$$
\noindent where $T_K:= \abs{\beta_K} + \abs{\beta'_K},$
  $\beta_K$,
$\beta'_K$ are
second and penultimate coefficients of  $J_K(t)$ and
 $s_K$  is the degree span of $J_K(t)$.
Furthermore, both bounds are sharp.
 \end{theorem}

\begin{example} {\rm{For $K=4_1$ we have $J_K(t)=t^{-2}-t^{-1}+ 1-t+ t^2$. Thus $T_K=1$ and $s_k=4$ and Theorem \ref{thm:cupjonesknots} gives $C(K)=2$.}}
\end{example}
The next result gives a strong connection of the Jones polynomial
to hyperbolic geometry as it estimates volume of hyperbolic alternating links in terms of coefficients of their Jones polynomials.
The result follows by work of Dasbach and Lin \cite{dfkls:determinant} and work
of Lackenby 
\cite{lackenby:alt-volume}.

\begin{theorem} \label{volume}Let $K$ be an alternating link  whose exterior admits a hyperbolic structure with volume  $\vol(S^3 \setminus K)$.
 Then we have
\begin{equation*}
\frac{\voct}{2} (T_K-2) \leq \vol(S^3 \setminus K) \leq 10\vtet (T_K - 1),
\end{equation*}
\noindent  where  $\voct=3.6638$ and $\vtet= 1.0149$.
\end{theorem}

To establish the lower bound  of Theorem \ref{volume} one looks at the state surfaces $S_A$, $S_B$ corresponding to a reduced alternating diagram $D(K)$:
Use $M_K \cut S_A$ to denote the complement in $M_K$ of a collar neighborhood of $S_A$. 
Jaco-Shalen-Johannson  theory \cite{jaco-shalen}   implies that there is  a canonical way to decompose $M_K \cut S_A$  along certain annuli
 into three types of pieces: (i) $I$--bundles over subsurfaces of $S_A$;  (ii) solid tori; and  (iii) the
 remaining pieces,  denoted by $\guts(M,S)$. On one hand,  by work Agol, Storm, and Thurston  \cite{AST:guts}, the quantity 
 $\abs{ \chi(\guts(M_K,S_A))}$ gives a lower bound 
 for the volume $ \vol(S^3 \setminus K)$. On the other hand, \cite{lackenby:alt-volume} shows that this quantity   is equal to  $1-\chi(\GRA)$, which by Corollary \ref{coeffs}
 is $\abs{\beta'_K}$. A similar consideration applies to the surface $S_B$ giving the lower bound of Theorem \ref{volume}. The approach was developed and  generalized to non-alternating links   in \cite{GutsBook}.

%\acknowledgement{The author is supported in part by NSF grants DMS-1404754 and DMS-1708249 }

\bibliographystyle{plain} \bibliography{biblio}
\end{document}